\documentclass{article}%
\usepackage{amsfonts}
\usepackage{amssymb}
\usepackage{amsmath}
\usepackage{graphicx}
\usepackage{cite}
\usepackage{pst-all}
\usepackage[english]{babel}
\usepackage[utf8]{inputenc}
\usepackage{authblk}
\usepackage{anysize}
\papersize{27.9cm}{21.5cm}
\marginsize{3.0cm}{3.0cm}{2.5cm}{2.5cm}

\providecommand{\U}[1]{\protect\rule{.1in}{.1in}}

\newtheorem{theorem}{Theorem}

\newtheorem{definition}[theorem]{Definition}
\newtheorem{example}[theorem]{Example}

\newtheorem{lemma}[theorem]{Lemma}

\newenvironment{proof}[1][Proof]{\noindent\textbf{#1.} }{\ \rule{0.5em}{0.5em}}

\newpsobject{grilla}{psgrid}{subgriddiv=1,griddots=10,gridlabels=6pt}

\begin{document}

\title{Recovering a Gaussian distribution from its minimum.}

\author[1,2]{Ricardo Restrepo \thanks{ricardo.restrepo@udea.edu.co}}
\author[2]{Carlos Marin}
\author[1,2]{Jose Solano}
\affil[1]{Departamento de Matematicas, Universidad de Antioquia}
\affil[2]{Research Division, Math Decision}



%

\maketitle


\begin{abstract}
Let $X=(X_1,X_2, X_3)$ be a Gaussian random vector such that $X\sim \mathcal{N} (0,\Sigma)$.  We consider the problem of determining the matrix $\Sigma$, up to permutation, based on the knowledge of the distribution of $X_{\mathrm{min}}:=\min(X_1, X_2, X_3)$. Particularly, we establish a connection between this identification problem and a geometric identification problem in the context of the theory of the circular radon transform.
\end{abstract}


\section{Introduction}

A {\em{Gaussian random vector\/}} is a vector valued random variable $X=(X_1,\dots, X_n)$ whose components $X_1,\dots , X_n$ are jointly Gaussian. The {\em{mean vector\/}} $\mu$, of a Gaussian random vector $X=(X_1,\dots, X_n)$ is defined by $\mu=(\mu_1,\dots, \mu_n)^T$, where $\mu_i= E(X_i)$ for $i = 1, \ldots, n$. Moreover, the {\em{correlation matrix\/}} $\Sigma$ of $X$ is an $n\times n$ matrix defined by $\Sigma:=(\Sigma_{ij})$, where $\Sigma_{ij}=\mathrm{Cov}(X_i,X_j)$ for $1\le i,j \le n$.
The statement ``$X$ is a Gaussian random vector with mean vector $\mu$ and correlation matrix $\Sigma$''  will be compactly written as $X\sim \mathcal{N} (\mu,\Sigma)$.

Suppose that $X=(X_1,\dots, X_n) \sim \mathcal{N} (0,\Sigma)$ and let $X_{\mathrm{min}}$ be the random variable defined by $X_{\mathrm{min}}= {\mathrm{min}} (X_1,\dots, X_n)$.
It is clear that the knowledge of $\Sigma$ allow us to determine the distribution of $X_{\mathrm{min}}$.
On the other hand, a natural question arises: Does the distribution of $X_{\mathrm{min}}$ determine the matrix $\Sigma$? In other words, is it possible to recover the matrix $\Sigma$ from the distribution of $X_{\mathrm{min}}$? The purpose of this paper is to solve this problem affirmatively, under the assumption that
$\boldsymbol{1}^t\Sigma^{-1}>0$, which in particular covers the case where the
correlations are negative \cite{DM06, DM07}.




This problem seems to have originated from an econometrics supply-demand problem posed to Anderson and Ghurye \cite{AG78}. Additionally, a similar set of problems were previously known in the context of competing and complimentary risks \cite{T75, BG78, BG80}. These kind of problems of identification have been the subject of interest to many authors, including \cite{AG78, EH0, EH01, DM01, DM07, EM99, MS90, R92}.  Particularly, in the same setting established above, the authors in \cite{BG78} studied this recovery problem under the hypothesis $\rho_{ij}\sigma_i < \sigma_j$. In \cite{EM99, DM01} it was studied the case of common correlations. And, in \cite{DM07, DM06} it was studied the case of non-negative correlations. The novelty of our approach consists in tackling the problem via a geometric approach, reducing it to a recovering problem in the context of circular radon transform. This allows us to significantly extend previous results. Moreover, from this approach much intuition is gained regarding the `backstage' geometric difficulty implicit in this and similar recovery problems, particularly improving previous results established in \cite{DM06} and \cite{DM07}.


This paper is organized as follows: Section \ref{sec:preliminares} is devoted to develop some preliminary geoemetric notions. Section \ref{sec:howrecoverSigma} contains the proof of the main result of the article, while Section \ref{sec:recoveringthetriangle} presents in detail the proof of Lemma \ref{lem:geo} which is a key component in the proof of the main theorem.


\section{Preliminaries}\label{sec:preliminares}

\subsection{Generalized square roots}\label{subset:squareroots}

\begin{definition}
\label{def:sqr}
We say that $N$ is a {\em{generalized square root\/}}  of a matrix $M$ (or, for the purpose of this article, just a \emph{square root} of $M$) if $NN^t=M$.
\end{definition}

Now, if $M$ is an $n\times n$ positive definite symmetric matrix, it has an eigendecomposition $P D P^t$, where $P$ is a unitary complex matrix and $D$ is a real diagonal matrix whose main diagonal contains the corresponding positive eigenvalues. Thus, the matrix $N=P D^{\frac{1}{2}}$ is a generalized square root of $M$, (where $D^{\frac{1}{2}}$ is a square root of $D$ in the `usual' sense).

Note that, given an orthogonal matrix $O$ and any square root $N$ of $M$, the matrix $NO$ is also a square root of $M$. Moreover, each square root for $M$ can be obtained by right multiplication of $N$ with an orthogonal matrix. In other words, the orthogonal group acts transitively by right multiplication on the set of square roots of $M$. 

Notice that a square root $N$ of a positive definite matrix $M$ is a natural change of variables between the Hilbert space induced by the internal product
$\left\langle u,v\right\rangle _{M}:=u^{t} M v$ and the Euclidean
Hilbert space endowed with the usual internal product $\left\langle u,v\right\rangle :=u^{t}v$,
in the sense that $\left\langle u,v\right\rangle _{M}=\left\langle
Nu,Nv\right\rangle $. A particular use of this fact appears when we consider a Gaussian random vector $X\sim
\mathcal{N}\left(0,\Sigma\right)$. In this case we have that $X\overset{\mathcal{D}}{=} NU$, where $N$ is any square root of $\Sigma$ and $U\sim \mathcal{N}\left(0,I\right)$.

We will impose an additional condition on the square roots of $\Sigma$:
If a square root $N$ of $\Sigma$ satisfies that $Ne_1$ is a positive multiple of the vector $\mathbf{1}=(1,\dots,1)^t$, then $N$ will be called a {\em{standard square root\/}} of $\Sigma$. Notice that, if this is the case,
$Ne_{1}=\kappa^{-1}\mathbf{1},$ with $\kappa:=\sqrt{\mathbf{1}^{t}
	\Sigma^{-1}\mathbf{1}}$.


\subsection{Cones and sections}\label{subset:conesandsections}
Let $A$ be an $n\times n$ non-singular matrix. The {\emph{positive cone\/}} associated with $A$ is defined as the set
$C_{A}:=\{u\in \mathbb{R}^n:Au\geq0\}.$\footnote{$\leq$ denotes the componentwise order, that is for a given vectors $v =(v_1,\dots, v_n)^t, w =(w_1,\dots, w_n)^t \in \mathbb{R}^n$,  $v\geq w$ sii $v_i \geq w_i$, for each $i: 1,\dots, n$.} A different representation for $C_A$ is given
by%
\[
C_{A}=\left\{\sum_{i=1}^n c_{i}v_{i}: c_{i}\geq0 \mbox{ for }  i= 1,\dots, n \right\}\text{,}%
\]
where, for $i=1, \ldots, n$, the vector $v_{i}$ is a positive multiple of $A^{-1}e_{i}$. The vectors $v_1,\dots, v_n$ are collectively called
the \emph{directions} of $C_A$.

In the following let us assume that $\Sigma^{-1}\boldsymbol{1}>0$ and let $N$ be a standard square root of $\Sigma$. Notice that
 $C_{N}$ (except for the origin) is completely contained in the half space $\{w\in \mathbb{R}^3:e_{1}^{t}w >0\}$. Indeed, if $0\ne u \in C_N$, then
\[
e_{1}^{t}u= \lambda \boldsymbol{1}^{t}\left(N^{-1}\right)^{t} N^{-1} Nu=\lambda \boldsymbol{1}^{t}%
\Sigma^{-1}Nu = \lambda (\Sigma ^{-1} \boldsymbol{1})^t Nu>0.
\]


In particular, since $e_1 \in C_N$, the requirement $\Sigma^{-1}\boldsymbol{1}>0$ allows us to choose the directions of $C_N$ in a such way that they have the form
$v_i=(1,\overline{\alpha_{i}})$, $i=1,\dots,n$. In such a case, the vectors $\overline{\alpha_{i}}$ correspond to the extremal points of the intersection of $C_{N}$ with the plane $
x_{1}=1$. We will denote by $T_N$ the convex set generated by the points $\overline{\alpha_{i}} \in\mathbb{R}^{n-1}$ and we will call it the {\em{section associated\/}} with the standard root $N$ of $\Sigma$.


\section{How to recover $\Sigma$ from the distribution of $X_{\mathrm{min}}$?}\label{sec:howrecoverSigma}

\subsection{Reducing the problem}\label{subset:reducingtheproblem}
Given a Gaussian random vector $X=(X_1,X_2, X_3) \sim \mathcal{N} (0,\Sigma)$,
we define $m_{\Sigma}$ as the tail distribution of
$X_{\mathrm{min}}=\min(X_1,X_2,X_3)$:
\[
m_{\Sigma}(t):=\operatorname*{Prob}\left( X_{\mathrm{min}}
\geq t\right),
\]
In the following, we will assume that $\Sigma^{-1}\boldsymbol{1}>0$.

\begin{lemma}
\label{pro:twoparameters}
Let $N$ be a standard square root of $\Sigma$. Then $\Sigma$ is uniquely
determined (up to a permutation), from the vertices of the
section $T_{N}$ and the parameter
$\kappa:=\sqrt{\boldsymbol{1}^{t}\Sigma^{-1}\boldsymbol{1}}$.
\end{lemma}

\begin{proof}
Let $w_i=(1,\alpha_i,\beta_i), i=1,2,3$ be the
vertices of the triangle $T_N$ and let $W$ be the matrix with columns
$w_{1},w_{2},w_{3}$. Define $\mu_{1},\mu_{2},\mu_{3}$ by the relation
$\left[  \mu_{1},\mu_{2},\mu_{3}\right] ^{t}:=\kappa W^{-1}e_{1}$ and
let ${\Lambda:={\rm{diag}}(\mu_{1}, \mu_{2}, \mu_{3})}$. We assert
that there exists a permutation matrix $P$ such that
\[
N=P^{t}\Lambda^{-1}W^{-1}.
\]
From this statement the lemma would clearly follow since we would have that
\[\Sigma=P^{t}\Lambda
^{-1}W^{-1}\left(  W^{-1}\right)  ^{t}\left(  \Lambda^{-1}\right)
^{t}P.
\]
In order to prove the assertion notice that, up to ordering and scaling,
the vectors $w_{1},w_{2},w_{3}$ are equal to the directions
$N^{-1}e_{1},N^{-1}e_{2},N^{-1}e_{3}$, thus
\[
(N^{-1}e_{1},N^{-1}e_{2},N^{-1}e_{3})=(\lambda_{\sigma(1)}w_{\sigma(1)}
,\lambda_{\sigma(2)}w_{\sigma(2)},\lambda_{\sigma(3)}w_{\sigma(3)})\]
for some permutation $\sigma$ and some $\lambda_{1},\lambda_{2},\lambda_{3}>0$.
This is equivalent to saying that $W\Lambda' P=N^{-1}$, where $P$ is the matrix permutation
associated to $\sigma$  and $\Lambda':= \mbox{diag}(\lambda_{1},\lambda_{2},\lambda_{3})$.

Consequently,
\[
W\left[  \lambda_{1},\lambda_{2},\lambda_{3}\right]  ^{t}=W\Lambda' 1=W\Lambda' P1=N^{-1}\boldsymbol{1}
=\kappa e_{1}.
\]
Thus $\left[  \lambda_{1},\lambda_{2},\lambda_{3}\right]  ^{t}:=\kappa W^{-1}e_{1}$,
implying that $\Lambda = \Lambda'$  from where the assertion follows.
\end{proof}
\vspace{1cm}

Let $N$ be a standard root of $\Sigma$. Notice that for any $u\in \mathbb{R}^3$  the condition $Nu\ge t\boldsymbol{1}$ is equivalent to ${N(u-t\kappa e_1) \ge 0}$, and therefore $\{u:Nu\geq t1\}=t\kappa e_{1}+C_{N}$. Then, we have that for $t>0$,
\begin{equation}\label{eq:geom}
m_{\Sigma}\left(  t\right)  =\operatorname*{Prob}\left(  X_{\min}  \geq t\right) =\mu\left(  t\kappa e_{1}+C_{N}\right),
\end{equation}
where $\mu$ is the standard Gaussian measure $\mathcal{N}\left(  0,I\right)
$ in $\mathbb{R}^{3}$. Moreover, if $\mu_{\sigma}$ stands for the Gaussian
measure $\mathcal{N}\left(  0,\sigma^{2}I\right)  $, it is the case that
\begin{equation}
m_{\Sigma}\left(  t\right)  =\mu_{1/(t\kappa)}\left(e_{1}+C_{N}\right)
\text{.} \label{eq:geom2}%
\end{equation}


\begin{lemma}
\label{pro:kappa}
As $t\to \infty$, $\ln m_{\Sigma}\left(  t\right) \sim -\frac
{t^{2}\kappa^{2}}{2}$. In particular,
\begin{equation}
\kappa^{2}=-\frac{1}{2}\lim_{t\rightarrow+\infty}\frac{\ln m_{\Sigma}\left(
t\right)  }{t^{2}}\text{.} \label{eq:kappa}%
\end{equation}
\end{lemma}
\begin{proof}
The rate function of the sequence of measures $\mu_{\sigma}$, where
$\sigma\rightarrow0$, is given by $I\left(  u\right)  =\left\Vert u\right\Vert
^{2}/2$, (see \cite{V88}). Therefore,
\[\ln\mu_{\sigma}\left(  e_{1}
+C_{N}\right)  \sim -\sigma^{-2} \inf_{u\in
\left(  e_{1}+C_{N}\right)}I(u);
\]
On the other hand, it is clear that
\[
\inf_{u\in
\left(  e_{1}+C_{N}\right)}I(u)=I(e_{1}) =\frac{1}{2}.
\]
Therefore, as $t\to \infty$, we have that
\[
\ln\left(  m_{\Sigma}\left(  t\right)  \right) = \ln\mu_{1/t\kappa}\left(
e_{1}+C_{N}\right)  \sim-\frac{t^{2}\kappa^{2}}{2}\text{.}%
\]
\end{proof}

\vspace{1cm}
We define the (two-dimensional) {\em{circular transform\/}} of a function
$f:\mathbb{R}^{2}\rightarrow\mathbb{R}$ by
\[
R_{f}\left(  \rho\right)  :=
{\displaystyle\int\limits_{0}^{2\pi}}
f\left(  \rho\cos\theta,\rho\sin\theta\right)  d\theta\text{.}
\]
For a measurable set $S\subseteq\mathbb{R}^{2}$, we define the circular transform by
\[
R_{S}\left(  \rho\right)  :=R_{I\left(  \cdot\in S\right)  }\left(
\rho\right)  =
{\displaystyle\int\limits_{0}^{2\pi}}
I_S(\rho\cos\theta,\rho\sin\theta) d\theta\text{,\quad}\rho>0\text{,}%
\]
where $I_S$ denotes the characteristic function of $S$. Notice that, since the angular measure is invariant under orthogonal transformations, the circular transform is also invariant under orthogonal transformations. In other words if $O$ is an orthogonal transformation, it is the case that $R_{f\circ O}=R_f$.
On the other hand, if $R_f = R_g$ for given functions $f$, $g$, it is not necessarily the case that $f \circ O = g$ for some orthogonal transformation (see example \ref{ex:norec}). In spite of this, we have a positive result in this  direction when $f$ and $g$ correspond to characteristic functions of triangles enclosing the origin (that is, such that $0$ belongs to its interior).

\begin{lemma}
	\label{lem:geo}If $T$ is a triangle \textbf{enclosing the origin}, then $T$ can be recovered (up to an orthogonal transformation), from $R_{T}$.
\end{lemma}
Due to its length, we relegate the proof of Lemma \ref{lem:geo} to Section \ref{sec:recoveringthetriangle}.
Although moderately technical, it relies on elementary geometry and basic linear algebra.
The condition that the triangles enclose the origin is necessary, as the following example shows:

\begin{example}
\label{ex:norec} Consider the triangles depicted in the figure below
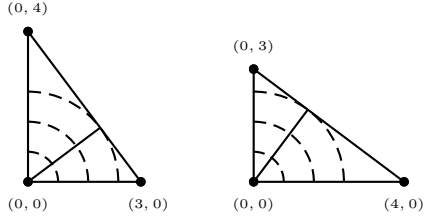
\[
\begin{pspicture}(-6,-1)(2,3)
\psline{*-*}(-5,2)(-5,0)
\psline{*-*}(-5,2)(-3.5,0)
\psline{*-*}(-5,0)(-3.5,0)
\rput(-5,-0.3){\tiny{$(0,0)$}}
\rput(-3.4,-0.3){\tiny{$(3,0)$}}
\rput(-5,2.3){\tiny{$(0,4)$}}
\psline(-5,0)(-4.04,0.72)
\psarc[linestyle=dashed](-5,0){0.4}{0}{90}
\psarc[linestyle=dashed](-5,0){0.8}{0}{90}
\psarc[linestyle=dashed](-5,0){1.2}{0}{90}

\psline{*-*}(-2,1.5)(0,0)
\psline{*-*}(-2,1.5)(-2,0)
\psline{*-*}(-2,0)(0,0)
\rput(-2,-0.3){\tiny{$(0,0)$}}
\rput(0,-0.3){\tiny{$(4,0)$}}
\rput(-2,1.8){\tiny{$(0,3)$}}
\psline(-2,0)(-1.28,0.96)
\psarc[linestyle=dashed](-2,0){0.4}{0}{90}
\psarc[linestyle=dashed](-2,0){0.8}{0}{90}
\psarc[linestyle=dashed](-2,0){1.2}{0}{90}
\end{pspicture}
\]
It is clear that the triangles have the same circular transform. Therefore, for the triangles
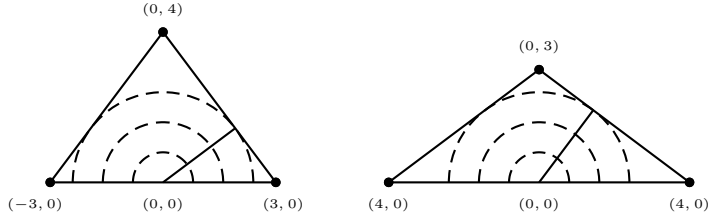
\[
\begin{pspicture}(-6,-1)(2,3)
\psline{*-*}(-5,2)(-6.5,0)
\psline{*-*}(-5,2)(-3.5,0)
\psline{*-*}(-6.5,0)(-3.5,0)
\rput(-5,-0.3){\tiny{$(0,0)$}}
\rput(-3.4,-0.3){\tiny{$(3,0)$}}
\rput(-6.7,-0.3){\tiny{$(-3,0)$}}
\rput(-5,2.3){\tiny{$(0,4)$}}
\psline(-5,0)(-4.04,0.72)
\psarc[linestyle=dashed](-5,0){0.4}{0}{180}
\psarc[linestyle=dashed](-5,0){0.8}{0}{180}
\psarc[linestyle=dashed](-5,0){1.2}{0}{180}

\psline{*-*}(0,1.5)(2,0)
\psline{*-*}(0,1.5)(-2,0)
\psline{*-*}(-2,0)(2,0)
\rput(0,-0.3){\tiny{$(0,0)$}}
\rput(2,-0.3){\tiny{$(4,0)$}}
\rput(-2,-0.3){\tiny{$(4,0)$}}
\rput(0,1.8){\tiny{$(0,3)$}}
\psline(0,0)(0.72,0.96)
\psarc[linestyle=dashed](0,0){0.4}{0}{180}
\psarc[linestyle=dashed](0,0){0.8}{0}{180}
\psarc[linestyle=dashed](0,0){1.2}{0}{180}
\end{pspicture}
\]
the circular transform is also the same. However, it is clear that they are not orthogonally equivalent.
\end{example}

\subsection{The main Theorem}\label{subset:main theorem}
Now we are able to state and prove the main result of this article:
\begin{theorem}
\label{th:min}
Suppose that  $X=(X_1,X_2, X_3) \sim \mathcal{N} (0,\Sigma)$. Let $X_{\mathrm{min}}$  be the random variable defined by $X_{\mathrm{min}}= {\mathrm{min}} (X_1,X_2, X_3)$ and assume that $\boldsymbol{1}^{t}\Sigma^{-1}>0$. Then, the distribution of
$X_{\mathrm{min}}$ uniquely determines $\Sigma$ up to permutation
equivalence. More exactly, if $Y=(Y_1,Y_2, Y_3) \sim \mathcal{N} (0,\Sigma_0)$ and $X_{\mathrm{min}} \overset{\mathcal{D}}{=} Y_{\mathrm{min}}$, there exists a permutation matrix $P$ such that $\Sigma_{0}=P\Sigma
P^{t}$.
\end{theorem}
Let us first prove the following lemma:
\begin{lemma}
\label{pro:rad} 
Let $N$ be a standard square root of $\Sigma$ (see definition \ref{def:sqr}) where $1^{t}\Sigma^{-1}>0$. Then $R_{T_{N}}$ is identifiable from $m_{\Sigma}\left(  t\right):=\Pr(X_{\mathrm{min}}\geq t)$. That is, $m_{\Sigma}\left(  t\right)$ uniquely determines $R_{T_N}$
\end{lemma}
\begin{proof}
From eq. \eqref{eq:geom}, we have that
\[
m_{\Sigma}\left(  t\right) =\mu\left(  \kappa te_{1}+C_{N}\right)
=\frac{1}{\left(  2\pi\right) ^{3/2}}
\iiint\limits_{ \kappa te_{1}+C_{N}}e^{-\frac{1}{2}\lVert u \rVert^2} \, du,
\]
where $\kappa^{2}=\boldsymbol{1}^{t}\Sigma^{-1}\boldsymbol{1}$. Equivalently,
\[
m_{\Sigma}\left(  t\right) = \frac{1}{\left(  2\pi\right) ^{3/2}}
{\displaystyle\int\limits_{\kappa t}^{+\infty}}
{\displaystyle\iint\limits_{\;\;\left( u_{1}-\kappa t\right)  T_{N}}}
e^{-\frac{1}{2}u_1^2 }e^{-\frac{1}{2}(u_2^2+u_3^2) }\, du_{2}\, du_{3} \, du_{1},
\]
Now, a change of variables leads to
\begin{equation*}
\begin{split}
m_{\Sigma}\left(  t\right) & = \frac{1}{\left(  2\pi\right) ^{3/2}}
{\displaystyle\int\limits_{0}^{+\infty}} e^{-\frac{1}{2}(x+\kappa t)^2 } \iint\limits_{xT_N} e^{-\frac{1}{2}(y^2+z^2) }\, dy \, dz  \, dx \vspace{0.5cms}\\
& = \frac{e^{-\frac{1}{2}\kappa^{2}t^{2}}}{\left(  2\pi\right) ^{3/2}}
{\displaystyle\int\limits_{0}^{+\infty}} e^{-\kappa t x}e^{-\frac{1}{2}x^{2}} \iint\limits_{xT_N} e^{-\frac{1}{2}(y^2+z^2) }\, dy\, dz  \, dx.\\
\end{split}
\end{equation*}
Therefore, if $\mathcal{L}$ denotes the Laplace transform and $h$ is the function defined by
\[
h(x):=e^{-\frac{1}{2}x^{2}} {\displaystyle\iint\limits_{xT_{N}}}
e^{-\frac{1}{2}\left( y^{2}+z^{2}\right)}\, dy \, dz,\]
we have that
\[
\left(2\pi\right)^{3/2}e^{\frac{1}{2}t^{2}}m_{\Sigma}\left(
t/\kappa\right)  =
{\displaystyle\int\limits_{0}^{+\infty}}
e^{-tx} \left [e^{-\frac{1}{2}x^{2}}
{\displaystyle\iint\limits_{xT_{N}}}
e^{-\frac{1}{2}\left( y^{2}+z^{2}\right)}\, d y \, d z \right]\, dx = \mathcal{L} \left(h\right)(t).
\]
Since the Laplace transform is injective over the functions of polynomial growth on $(0,\infty)$, in particular, we have
\[
\mathcal{L}^{-1}\left(  \left(  2\pi\right)  ^{3/2}e^{\frac{1}{2}t^{2}%
}m_{\Sigma}\left(  t/\kappa\right)  \right)(\sqrt{2x}) = h(\sqrt{2x}) = e^{-x}
{\displaystyle\iint\limits_{\sqrt{2x}T_{N}}}
e^{-\frac{1}{2}\left( y^{2}+z^{2}\right)}\, d y \, d z,
\]
or equivalently,
\begin{equation}
\frac{e^{x}}{x}
\mathcal{L}^{-1}\left(  \left(  2\pi\right)  ^{3/2}e^{\frac{1}{2}t^{2}
}m_{\Sigma}\left(  t/\kappa\right)  \right)(\sqrt{2x}) =
\frac{1}{x}
{\displaystyle\iint\limits_{\sqrt{2x}T_N}}
e^{-\frac{1}{2}\left(  y^{2}+z^{2}\right)  }\, d y \, d z. \label{eq:lap2}
\end{equation}
On the other hand, it is clear that
\[
\begin{split}
\frac{1}{x}
\iint\limits_{\sqrt{2x}T_N}
e^{-\frac{1}{2}(y^{2}+z^{2})}\, dy \, dz& =2 {\displaystyle\iint\limits_{T_N}}
e^{-x(y^{2}+z^{2})}\, dy\, dz\\
&=2 \iint\limits_{(0,\infty)\times (0,2\pi)} e^{-x\rho^2} I_{T_N}(\rho\cos\theta,\rho\sin\theta)   \rho \, d \theta \, d \rho;\\
&= {\displaystyle\int\limits_{0}^{+\infty}}
e^{-x\rho}
\left[{\displaystyle\int\limits_{0}^{2\pi}}
I_{T_N}(\sqrt{\rho}\cos\theta,\sqrt{\rho}\sin\theta)  d\theta \right]d\rho\\
& = \mathcal{L}\left(g\right)(x),
\end{split}
\]
where $g$ denotes the function $g(\rho):=\int\limits_{0}^{2\pi}
I_{T_N} (\sqrt{\rho}\cos\theta,\sqrt{\rho}\sin\theta) \, d\theta$.\\
Therefore, it follows that
\begin{equation}
\mathcal{L}^{-1}\left(\frac{1}{x}
\iint\limits_{\sqrt{2x}T_N}
e^{-\frac{1}{2}\left(  y^{2}+z^{2}\right)}\, dy \, dz\right)(\rho^{2}) = \int\limits_{0}^{2\pi}
I_{T_N}(\rho\cos\theta,\rho\sin\theta) \, d\theta \text{.} \label{eq:lap3}%
\end{equation}
Finally, by combining Eqs. (\ref{eq:lap2}) and
(\ref{eq:lap3}) we obtain that
\begin{equation}
R_{T_{N}}\left(  \rho\right) =\mathcal{L}^{-1}\left(\frac{e^{x}}{\sqrt{2x}}\mathcal{L}^{-1}\left(  \left(  2\pi\right)  ^{3/2}e^{\frac{1}{2}t^{2}%
}m_{\Sigma}\left(  t/\kappa\right)  \right)  \left(\sqrt{2x}\right)
\right)  \left(  \rho^{2}\right)  \text{.} \label{eq:lap1}%
\end{equation}
Therefore,  equation \eqref{eq:lap1} give us an expression for $R_{T_{N}}$ in terms of $m_{\Sigma}\left(
t\right) $ and $\kappa$.  However, from lemma \ref{pro:kappa} we know that $\kappa$ is recoverable
from $m_{\Sigma}\left(  t\right) $, thus the
rigth hand side of \eqref{eq:lap1} is determined from $m_{\Sigma}\left(
t\right) $, and the result follows.
\end{proof}

\vspace{1cm}

\begin{proof}
[Proof of Theorem \ref{th:min}]
From lemma \ref{pro:kappa}, $\kappa$ is
recoverable from $m_{\Sigma}\left(  t\right) $ using eq.
\eqref{eq:kappa}.
Moreover, from lemma \ref{pro:rad}, we can recover
$R_{T_{N}}$ from $m_{\Sigma}\left(  t\right) $ using the formula stated in
eq. (\ref{eq:lap1}), where $N$ is a standard square root of
$\Sigma$.
Therefore, using Lemma \ref{lem:geo}, we can recover $T_{N}$ up to an orthogonal
transformation, that is, we can recover $OT_{N}$ where $O$ is an (unknown)
orthogonal transformation in $\mathbb{R}^{2}$.
Now, we define $\widetilde{O}=
\begin{bmatrix}
1 & 0\\
0 & O
\end{bmatrix}
$.
Notice that $OT_{N}=T_{N\widetilde{O}}$, so that
$OT_{N}$ is the section associated with the standard root $\widetilde{N}:=N\widetilde
{O}$. Therefore, since we can recover $T_{\widetilde{N}}$ and $\kappa$, from lemma
\ref{pro:twoparameters} we can recover $\Sigma$ up to permutation equivalence.
\end{proof}


\section{The circular transform of a triangle}\label{sec:recoveringthetriangle}

We can think of the circular transform as a systematic `scan' that recognizes
the mass at distance $\rho$ from the origin. When the scan encounters an abrupt change of media, its smoothness is momentarily lost. 
For instance, in order to recover an acute triangle $T$ containing the origin, by detecting changes
in the smoothness of $R_{T}$ 
we can determine the distance from the origin to the sides and vertices of the triangle and then rely on a geometric construction to recover $T$. However, we take a detour from this approach and we use a more concise tool. Namely, that the circular transform of some `basic triangles' form a linearly independent set (lemma
\ref{lem:independent}). Then, we use such basic triangles as
building blocks to recover more complex geometric figures, in particular, any triangle containing the origin.

\subsection{Parametric form of a triangle}

Let $T$ be a triangle with vertices $x,y,z\in\mathbb{R}^{2}$ and such that $0$ belongs to
the interior of $T$. Also, let $n_{xy},n_{yz},n_{xz}$ be the closest points to
the origin from the lines $\overrightarrow{xy}$, $\overrightarrow{yz}$ and
$\overrightarrow{xz}$ respectively (we will call them the \emph{heights} of
$T$). We define the \emph{parametric form} of the triangle $T$  as the
\textbf{ordered} sequence of distances $\left(  \left\Vert x\right\Vert
,\left\Vert n_{xy}\right\Vert ,\left\Vert y\right\Vert ,\left\Vert
n_{yz}\right\Vert ,\left\Vert z\right\Vert ,\left\Vert n_{xz}\right\Vert
\right)$. It is an easy geometric fact that this list determines the
triangle $T$ up to rotation.


As an intermediate step we require the following lemma, which states that,
if the terms of the parametric form can be recovered \emph{by pairs}, then the
parametric form of the triangle can be recovered.

\begin{lemma}
\label{lem:pairs}
Consider a triangle $T$ with parametric form $\left(  \left\Vert x\right\Vert
,\left\Vert n_{xy}\right\Vert ,\left\Vert y\right\Vert ,\left\Vert
n_{yz}\right\Vert ,\left\Vert z\right\Vert ,\left\Vert n_{xz}\right\Vert
\right)$. Let $\alpha, \beta,\gamma,\eta_{1},\eta_{2},\eta_{3}$ satisfy
\[
\left\{
\begin{matrix}
(\eta_{1},\alpha),&(\eta_{1},\beta), &(\eta_{2},\beta),\\
(\eta_{2},\gamma),&(\eta_{3},\gamma), &(\eta_{3},\alpha)\\
\end{matrix}
\right\}
=
\left\{
\begin{matrix}
\left(\left\Vert n_{xy}\right\Vert,\left\Vert x\right\Vert\right),&(\left\Vert n_{xy}\right\Vert ,\left\Vert y\right\Vert ), &( \left\Vert n_{yz}\right\Vert ,\left\Vert y\right\Vert),\\
(\left\Vert n_{yz}\right\Vert ,\left\Vert z\right\Vert),&( \left\Vert n_{xz}\right\Vert ,\left\Vert z\right\Vert), &(\left\Vert n_{xz}\right\Vert ,\left\Vert x\right\Vert)\\
\end{matrix}
\right\}
\]
Then, a parametric form of $T$ is given by
\[
\left(  \alpha,\eta_{1},\beta,\eta_{2},\gamma,\eta_{3}\right)  \text{.}%
\]
\end{lemma}
\begin{proof}
Without loss of generality, we can assume that $\alpha=\left\Vert x\right\Vert
$, $\beta=\left\Vert y\right\Vert $ and $\gamma=\left\Vert z\right\Vert $. If it is the case
that $\alpha$, $\beta$ and $\gamma$ are different numbers, then it is clear
that $\left\Vert n_{xy}\right\Vert =\eta_{1}$, $\left\Vert n_{yz}\right\Vert
=\eta_{2}$ and $\left\Vert n_{xz}\right\Vert =\eta_{3}$. On the other hand, if
two of these numbers are equal, say $\alpha=\beta\neq\gamma$, then necessarily
$\left\Vert n_{xy}\right\Vert =\eta_{1}$ and $\left\{  \left\Vert
n_{yz}\right\Vert ,\left\Vert n_{xz}\right\Vert \right\}  =\left\{  \eta
_{2},\eta_{3}\right\}  $ where the two possible choices lead to isomorphic
triangles, therefore, by appropriately interchanging $y$ and $z$, we have that
$\left\Vert n_{yz}\right\Vert =\eta_{2}$ and $\left\Vert n_{xz}\right\Vert
=\eta_{3}$. Finally, if it is the case that $\alpha=\beta=\gamma$, then we
have that $\left\{  \left\Vert n_{xy}\right\Vert ,\left\Vert n_{yz}\right\Vert
,\left\Vert n_{xz}\right\Vert \right\}  =\left\{  \eta_{1},\eta_{2},\eta
_{3}\right\}  $ and in fact any assignation leads to isomorphic triangles. Therefore, appropriately interchanging $x$, $y$ and $z$, we have that
$\left\Vert n_{xy}\right\Vert =\eta_{1}$, $\left\Vert n_{yz}\right\Vert
=\eta_{2}$, $\left\Vert n_{xz}\right\Vert =\eta_{3}$.
\end{proof}

\subsection{Some properties of the circular transform}
A family of functions $\left\{  f_{\alpha}\left(
x\right)  \right\} _{\alpha \in \Lambda}$ is said to be {\em{linearly independent\/}} in an interval $I$
if for any {\em{finite\/}} set $F\subset \Lambda$,  $\sum\nolimits_{\alpha\in F}c_{\alpha
}f_{\alpha}\left(  x\right)  =0$ for all $x\in I$ implies that
$c_{\alpha}=0$ for all $\alpha\in F$.\\

For any $a, b>0$, we define the triangle
\[
T\left(  a,b\right)
:=\mbox{conv}\left( (0, 0),\left(  a,0\right)  ,\left(  a,\sqrt
{b^{2}-a^{2}}\right)  \right)
\]
where \emph{$\mbox{conv}$} stands for the convex closure in $\mathbf{R}^2$. We also define the function $\Phi_{a,b}:(0,+\infty)\to \mathbb{R}$ by
\[
\Phi_{a,b}\left(  \rho\right)
:=R_{T\left(  a,b\right)  }\left(  \rho\right).
\]

\begin{lemma}
\label{lem:independent}The set of functions $\left\{  \Phi_{a,b}\right\}
_{a,b>0}$ is linearly independent in $\left(  0,+\infty\right)  $.
\end{lemma}

We first prove the following lemma.

\begin{lemma}\label{lemma:1}
Consider an increasing sequence of real numbers $0<a_{0}<\cdots<a_{m}$. Then, for every $x\in \left(0,1/a_{m}^{2}\right)$ and every $i, j = 1,\ldots, m$, $i \neq j$, it is the case that
\[
a_i^{2}\left(1-a_i^{2}x\right) \ne a_j^{2}\left(1-a_j^{2}x\right)
\]
\end{lemma}
\begin{proof}
Let $x\in \left(0,1/a_{m}^{2}\right)$ be fixed. Since $0<x<1/a_{m}^{2} \leq 2/a_{i}^{2}$, for all $i=1,\ldots,m$, we have that $a_i\in\left(0,\sqrt{2/x}\right)$ for $i=0,\ldots,m$.  The result follows from the fact that the function $f(t):= t^{2}\left(  1-t^{2}x\right)$ is increasing in $(  0,\sqrt{2/x})$.
\end{proof}

\vspace{0.4cm}
\begin{proof}[Proof of Lemma \ref{lem:independent}]
Given $0<a_{0}<\cdots<a_{m}$, consider the functions $\left\{\psi_{a_{i}}\left(  x\right)  \right\}  _{i=0}^{m}$, where
\[
\psi_{a}\left(  x\right)
:=\left(1-a^{2}x\right)^{-1/2}.
\]
The Wronskian of these functions in $(0,1/a_{m}^{2})$ is given by
\[
W\left(\psi_{a_{0}},\ldots,\psi_{a_{m}}\right)  =C\left\vert
\begin{array}
[c]{ccc}%
1 & \cdots & 1\\
a_{0}^{2}\left(  1-a_{0}^{2}x\right)  ^{-1}& \cdots & a_{m}^{2}\left(
1-a_{m}^{2}x\right)  ^{-1}\\
\vdots & \ddots & \vdots\\
a_{0}^{2m}\left(  1-a_{0}^{2}x\right)  ^{-m} & \cdots & a_{m}^{2m}\left(
1-a_{m}^{2}x\right)  ^{-m}%
\end{array}
\right\vert \text{,}%
\]
where
\[
C=\prod\nolimits_{i=0}^{m}
\dfrac{\left(  2i\right)  !}{i!4^{i}}
\prod\nolimits_{i=0}^{m}
\left(  1-a_{i}^{2}x\right)  ^{-1/2}.
\]
The matrix in the expression of $W\left(\psi_{a_{0}},\ldots,\psi_{a_{m}}\right)$ is a
Vandermonde matrix, with generators $a_{i}^{2}\left(1-a_{i}
^{2}x\right)^{-1/2}$, for $i= 0,\ldots, m$ . It follows from Lemma \ref{lemma:1} that these generators are all different.
Therefore the Wronskian
$W\left(\psi_{a_{0}},\ldots,\psi_{a_{m}}\right)$ is non zero in $(0,1/a_{m}^{2})$. Consequently, for any fixed interval $I\subseteq (  0,1/a_{m}^{2})  $, the functions $\left\{
\psi_{a_{i}}\left(  x\right)  \right\}  _{i=0}^{m}$ are linearly independent
in $I$. Therefore, the functions $\left\{  a_{i}x \, \psi_{a_{i}}\left(  x\right)  \right\}  _{i=0}^{m}$ are linearly independent in $I \subseteq\left(  0,1/a_{m}^{2}\right)$.
In consequence, the change of variables  $x=1/\rho^{2}$, produces linearly independent functions in $I\subseteq\left( a_{m},+\infty\right)$, given by
\[
\frac{a_{i}}
{\rho\sqrt{\rho^{2}-a_{i}^{2}}}, \,\,i=0,\dots, m.
\]
Moreover, their antiderivatives $\left\{  \arccos\left(  a_{i}/\rho\right)  +\gamma
_{i}\right\}  _{i=0}^{m}$, where $\gamma_{0},\ldots,\gamma_{m}$ are arbitrary
constants, will be also linearly independent in $I\subseteq\left( a_{m}%
,+\infty\right) $.

Now, in order to prove that the functions $\left\{  \Phi_{a,b}\right\}_{a,b>0}$ are linearly independent
in $\left(  0,+\infty\right)$, we assume that
\[
\label{eq:mindep}
\sum\nolimits_{i=0}^{l} \sum\nolimits_{j=0}^{m_{i}}
c_{i,j}\Phi_{a_{i,j},b_{i}}\left(  \rho\right)  =0,
\,\, \forall \rho>0\text{;}
\]
where $0<a_{i,0}<\cdots<a_{i,m_{i}}$ for $i=0,\ldots,l$, $0<b_{1}<\cdots
<b_{l}$ and all the coefficients $c_{i,j}$ are nonzero (that is, we assume the
existence of a nontrivial minimal dependent set).

Now, eq. (\ref{eq:mindep}) implies that $\Phi_{a_{i,j},b_{i}}\left(
\rho\right)  =0$, whenever $i<l$ and $\rho\in\left(  b_{l-1},b_{l}\right)  $, therefore
\begin{equation}
\label{eq:minimal}
{\textstyle\sum\nolimits_{j=0}^{m_{l}}}
c_{l,j}\Phi_{a_{l,j},b_{l}}\left(  \rho\right)  =0\text{\quad,\quad} \forall \rho
\in\left(  b_{l-1},b_{l}\right)  \text{.}%
\end{equation}

Now, since
\[
\Phi_{a,b}\left(\rho\right)=\arccos\left(a/b\right)  I\left(\rho\leq
a\right)+\left[\arccos\left(a/b\right)-\arccos\left(a/\rho\right)
\right]  I\left(  a\leq\rho<b\right),%
\]

the eq. \ref{eq:minimal} restricted to the interval $\left(  a_{l,m_{l}%
},b_{l}\right)  $ takes the form
\[
{\textstyle\sum\nolimits_{j=0}^{m_{l}}}
c_{l,j}\left[  \arccos\left(  a_{l,j}/\rho\right)  -\arccos\left(
a_{l,j}/b_{l}\right)  \right]  =0,\text{\quad for all }\rho\in\left(
a_{l,m_{l}},b_{l}\right)  \text{,}%
\]
which contradicts the fact that the set $\left\{  \arccos\left(  a_{l,j}%
/\rho\right)  +\gamma_{j}\right\} _{j=0}^{m}$, with $\gamma_{j}%
=-\arccos\left(  a_{l,j}/b_{l}\right)  $, $j=0, \ldots, m$ is linearly independent in the
interval $\left(  a_{l,m_{l}},b_{l}\right)  $. 
\end{proof}

\smallskip\medskip

\subsection{Decomposing the triangle}
In order to prove Lemma \ref{lem:geo} we will distinguish the following three cases:
\begin{itemize}
\item \textbf{Case I:} All the heights of $T$ lie in the interior of the
sides. That is $n_{xy}=s_{1}x+\left(  1-s_{1}\right)  y$, $n_{yz}%
=s_{2}y+\left(  1-s_{2}\right)  z$ and $n_{xz}=s_{3}x+\left(
1-s_{3}\right)  z$, for some $s_{1},s_{2},s_{3}\in\left(  0,1\right)  $.

\item \textbf{Case II:} There is one height that does not lie in the interior
of the corresponding side. Say, withouth loss of generality, $n_{xy}=sx+\left(  1-s\right)  y$,
where $s>1$.

\item \textbf{Case III:} There is one height that lies over a vertex. Say,
withouth loss of generality, $n_{xy}=x$.
\end{itemize}

We define the \emph{basic subtriangles} of the triangle $T$ as the following six right
triangles:%
\begin{align*}
T_{1}  &  =\operatorname*{conv}\left(  0,x,n_{xy}\right)  \text{, }%
T_{2}=\operatorname*{conv}\left(  0,y,n_{xy}\right)  \text{, }%
T_{3}=\operatorname*{conv}\left(  0,y,n_{yz}\right)  \text{, }\\
T_{4}  &  =\operatorname*{conv}\left(  0,z,n_{yz}\right)  \text{, }%
T_{5}=\operatorname*{conv}\left(  0,z,n_{xz}\right)  \text{, }%
T_{6}=\operatorname*{conv}\left(  0,x,n_{xz}\right).
\end{align*}
Notice that these triangles are non degenerate (that is, of positive area),
except in case III, in which $T_{1}$ is a degenerate triangle. We should also point out the following elementary facts:\\

Let $I_A:\mathbf{R}^2 \to \{0, 1\}$ denotes the characteristic function on the set $A\subseteq \mathbf{R}^2$.
Now,
\begin{itemize}
\item In \textbf{case I},
\begin{equation}
I_T ={\textstyle\sum\nolimits_{i=1}^{6}}I_{T_i}  \text{.} \label{eq:char1}%
\end{equation}
\item In \textbf{case II},
\begin{equation}
I_{T}  =-I_{T_1} + {\textstyle\sum\nolimits_{i=2}^{6}}I_{T_i}, \label{eq:char2}
\end{equation}
and the subtriangle $T_{1}$ is not congruent to any other of the subtriangles.

\item In \textbf{case III},
\begin{equation}
I_{T}  = {\textstyle\sum\nolimits_{i=2}^{6}} I_{T_i}  \text{.} \label{eq:char3}%
\end{equation}
\end{itemize}

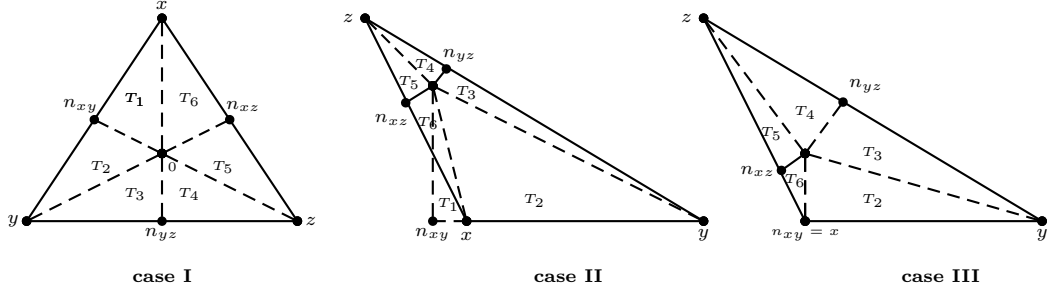
\[
\psset{unit=0.9}
\begin{pspicture}(-8,-1)(8,3)
\psline[linestyle=dashed]{*-*}(-5,1)(-5,3)
\psline{*-*}(-5,3)(-7,0)
\psline[linestyle=dashed]{*-*}(-5,1)(-7,0)
\psline{*-*}(-3,0)(-7,0)
\psline[linestyle=dashed]{*-*}(-5,1)(-5,0)
\psline[linestyle=dashed]{*-*}(-5,1)(-3,0)
\psline{*-*}(-3,0)(-5,3)
\psline[linestyle=dashed]{*-*}(-5,1)(-4,1.5)
\psline[linestyle=dashed]{*-*}(-5,1)(-6,1.5)
\rput(-5,3.2){\scriptsize{$x$}}
\rput(-7.2,0){\scriptsize{$y$}}
\rput(-2.8,0){\scriptsize{$z$}}
\rput(-5,-0.2){\scriptsize{$n_{yz}$}}
\rput(-3.8,1.7){\scriptsize{$n_{xz}$}}
\rput(-6.2,1.7){\scriptsize{$n_{xy}$}}
\rput(-4.85,0.8){\tiny{$0$}}
\rput(-5.4,1.8){\tiny{$T_1$}}
\rput(-5.9,0.8){\tiny{$T_2$}}
\rput(-5.4,0.4){\tiny{$T_3$}}
\rput(-4.6,0.4){\tiny{$T_4$}}
\rput(-4.1,0.8){\tiny{$T_5$}}
\rput(-4.6,1.8){\tiny{$T_6$}}
\rput(-5,-0.8){\scriptsize{\textbf{case I}}}

\psline{*-*}(-0.5,0)(3,0)
\psline{*-*}(-2,3)(3,0)
\psline{*-*}(-2,3)(-0.5,0)
\psline[linestyle=dashed]{*-*}(-2,3)(-1,2)
\psline[linestyle=dashed]{*-*}(-1,2)(-0.5,0)
\psline[linestyle=dashed]{*-*}(-1,2)(3,0)
\psline[linestyle=dashed]{*-*}(-1,2)(-1,0)
\psline[linestyle=dashed]{*-*}(-1,2)(-1.4,1.75)
\psline[linestyle=dashed]{*-*}(-1,2)(-0.8,2.25)
\rput(-0.6,2.45){\scriptsize{$n_{yz}$}}
\rput(-1.6,1.5){\scriptsize{$n_{xz}$}}
\rput(-1,-0.2){\scriptsize{$n_{xy}$}}
\rput(-2.25,3){\scriptsize{$z$}}
\rput(-0.5,-0.22){\scriptsize{$x$}}
\rput(3,-0.2){\scriptsize{$y$}}
\rput(-5.4,1.8){\tiny{$T_1$}}
\psline[linestyle=dashed](-0.5,0)(-1,0)
\rput(-0.77,0.25){\tiny{$T_1$}}
\rput(0.5,0.3){\tiny{$T_2$}}
\rput(-0.5,1.9){\tiny{$T_3$}}
\rput(-1.11,2.3){\tiny{$T_4$}}
\rput(-1.35,2.05){\tiny{$T_5$}}
\rput(-1.07,1.45){\tiny{$T_6$}}
\rput(1,-0.8){\scriptsize{\textbf{case II}}}

\psline{*-*}(4.5,0)(8,0)
\psline{*-*}(3,3)(8,0)
\psline{*-*}(3,3)(4.5,0)
\psline[linestyle=dashed]{*-*}(3,3)(4.5,1)
\psline[linestyle=dashed]{*-*}(4.5,1)(4.5,0)
\psline[linestyle=dashed]{*-*}(4.5,1)(8,0)
\psline[linestyle=dashed]{*-*}(4.5,1)(4.5,0)
\psline[linestyle=dashed]{*-*}(4.5,1)(5.06,1.76)
\psline[linestyle=dashed]{*-*}(4.5,1)(4.15,0.75)
\rput(4.5,-0.2){\tiny{$n_{xy}=x$}}
\rput(5.4,2){\scriptsize{$n_{yz}$}}
\rput(3.8,0.74){\scriptsize{$n_{xz}$}}
\rput(8,-0.2){\scriptsize{$y$}}
\rput(2.75,3){\scriptsize{$z$}}
\rput(5.5,0.3){\tiny{$T_2$}}
\rput(5.5,1){\tiny{$T_3$}}
\rput(4.5,1.6){\tiny{$T_4$}}
\rput(4,1.3){\tiny{$T_5$}}
\rput(4.35,0.6){\tiny{$T_6$}}
\rput(6.5,-0.8){\scriptsize{\textbf{case III}}}
\end{pspicture}
\]

Given a basis $\mathcal{B}=\{v_i\}_{i=1}^d$ of a vector space and an expression

\begin{equation}
	\label{eq:minexp}
w = \sum_{j} \beta_{t_j} v_{t_j}
\end{equation}

where $\{t_j\}_j$ is an arbitrary sequence in $\{1,\ldots,d\}$, we say that eq. (\ref{eq:minexp})
is \emph{irreducible} if $\sum_{i=1}^d |\alpha_i| = \sum_j |\beta_{t_j}|$ where
$\alpha_1,\ldots,\alpha_d$ are the coefficients of $w$ in the basis $\mathcal{B}$.

\begin{lemma}
\label{lem:span}$R_{T}$ belongs to the linear span of $\left\{  \Phi
_{a,b}\right\}  _{a,b>0}$. More exactly, if $\eta_1 := n_{xy}$, $\eta_2 :=n_{yz}$, $\eta_3:=n_{xz}$,
\begin{enumerate}
\item In case I,
\[
R_{T}=\Phi_{\left\Vert \eta_{1}\right\Vert ,\left\Vert x\right\Vert }%
+\Phi_{\left\Vert \eta_{1}\right\Vert ,\left\Vert y\right\Vert }%
+\Phi_{\left\Vert \eta_{2}\right\Vert ,\left\Vert y\right\Vert }%
+\Phi_{\left\Vert \eta_{2}\right\Vert ,\left\Vert z\right\Vert }%
+\Phi_{\left\Vert \eta_{3}\right\Vert ,\left\Vert x\right\Vert }%
+\Phi_{\left\Vert \eta_{3}\right\Vert ,\left\Vert z\right\Vert };
\]
\item In case II,
\[
R_{T}=-\Phi_{\left\Vert \eta_{1}\right\Vert ,\left\Vert x\right\Vert }%
+\Phi_{\left\Vert \eta_{1}\right\Vert ,\left\Vert y\right\Vert }%
+\Phi_{\left\Vert \eta_{2}\right\Vert ,\left\Vert y\right\Vert }%
+\Phi_{\left\Vert \eta_{2}\right\Vert ,\left\Vert z\right\Vert }%
+\Phi_{\left\Vert \eta_{3}\right\Vert ,\left\Vert x\right\Vert }%
+\Phi_{\left\Vert \eta_{3}\right\Vert ,\left\Vert z\right\Vert };
\]

\item In case III,%
\[
R_{T}=\Phi_{\left\Vert \eta_{1}\right\Vert ,\left\Vert y\right\Vert }%
+\Phi_{\left\Vert \eta_{2}\right\Vert ,\left\Vert y\right\Vert }%
+\Phi_{\left\Vert \eta_{2}\right\Vert ,\left\Vert z\right\Vert }%
+\Phi_{\left\Vert \eta_{3}\right\Vert ,\left\Vert x\right\Vert }%
+\Phi_{\left\Vert \eta_{3}\right\Vert ,\left\Vert z\right\Vert }\text{.}%
\]
\end{enumerate}
Moreover, all of the previous linear representations are irreducible.
\end{lemma}

\begin{proof}
By using the linearity of the circular
transform, all three formulas follow directly from Eqs. (\ref{eq:char1}), (\ref{eq:char2}%
) and (\ref{eq:char3}) respectively. The irreducibility is clear in cases I and III
due to the positivity of the coefficients. On the other hand, to show the irreducibility in case II we must show that the function $\Phi_{\left\Vert \eta_{1}\right\Vert
,\left\Vert x\right\Vert }$ is not in the set
\[
\left\{  \Phi_{\left\Vert
\eta_{1}\right\Vert ,\left\Vert y\right\Vert },\Phi_{\left\Vert \eta
_{2}\right\Vert ,\left\Vert y\right\Vert },\Phi_{\left\Vert \eta
_{2}\right\Vert ,\left\Vert z\right\Vert },\Phi_{\left\Vert \eta
_{3}\right\Vert ,\left\Vert x\right\Vert },\Phi_{\left\Vert \eta
_{3}\right\Vert ,\left\Vert z\right\Vert }\right\} ,
\] Indeed this is true since, in this case, the triangle $T_{1}$ is not congruent to
any other subtriangle.
\end{proof}

\subsection{Recovering the triangle}
\begin{proof}[Proof of Lemma \ref{lem:geo}] From Lemma \ref{lem:span}, $R_{T}$ can be expressed in basis $\left\{  \Phi
_{a,b}\right\}  _{a,b>0}$, say
\begin{equation}
R_{T}=%
{\textstyle\sum}
c_{a,b}\Phi_{a,b}\text{.} \label{eq:rep1}%
\end{equation}
Moreover, depending if $%
{\textstyle\sum}
c_{a,b}$ is equal to $6,4$ or $5$, we can distinguish if we are in case I, II
or III respectively. Notice also that the
representation in eq. (\ref{eq:rep1}) can be expressed as:%
\[
R_{T}=c\Phi_{a_{1},b_{1}}+\Phi_{a_{1},b_{2}}+\Phi_{a_{2},b_{2}}+\Phi
_{a_{2},b_{3}}+\Phi_{a_{3},b_{3}}+\Phi_{a_{3},b_{1}}\text{,}%
\]
for some $a_{i},b_{i}$, $i=1,\ldots,3$ and where $c$ is equal to $1$, $-1$ or $0$,
depending if we are in case I, II or III. We claim that 
$\left(b_{1},a_{1},b_{2},a_{2},b_{3},a_{3}\right)  $ is a parametric representation
of $T$. To see this, notice that in the cases I and II, from the uniqueness
of the linear representation we have that
\[
\left\{
\begin{array}{ccc}
\left(a_{1},b_{1}\right),&\left(a_{1},b_{2}\right), &\left(a_{2},b_{2}\right),\\
\left(  a_{2},b_{3}\right),&\left(a_{3},b_{3}\right),&\left(  a_{3},b_{1}\right)\\
\end{array}
\right\}  =\left\{
\begin{array}{ccc}
\left( \left\Vert n_{xy}\right\Vert ,\left\Vert x\right\Vert \right),&\left(\left\Vert n_{xy}\right\Vert,\left\Vert y\right\Vert \right),&\left(  \left\Vert n_{yz}\right\Vert ,\left\Vert y\right\Vert \right),\\
\left(  \left\Vert n_{yz}\right\Vert ,\left\Vert z\right\Vert \right)
, &\left( \left\Vert n_{xz}\right\Vert ,\left\Vert z\right\Vert \right), &\left(  \left\Vert n_{xz}\right\Vert ,\left\Vert x\right\Vert \right)\\
\end{array}
\right\}  \text{.}%
\]
Then, from Lemma \ref{lem:pairs}, it follows that $\left(  b_{1},a_{1}%
,b_{2},a_{2},b_{3},a_{3}\right)  $ is a parametric representation of $T$. In
case III the claim follows the same argument, by adding the
missing pair $\left(  a_{1},a_{1}\right)  $.
\end{proof}

\end{document}